\documentclass[12pt]{amsart}
\usepackage{amsfonts, amssymb, latexsym, array}
\usepackage{fullpage,color}
\usepackage{verbatim,euscript}

\numberwithin{equation}{section}

\usepackage[colorlinks=true,linkcolor=blue,urlcolor=violet,citecolor=magenta]{hyperref}

\input {cyracc.def}
\font\tencyr=wncyr10 
\font\tencyi=wncyi10 
\font\tencysc=wncysc10 
\def\rus{\tencyr\cyracc}
\def\rusi{\tencyi\cyracc}
\def\rusc{\tencysc\cyracc}

\catcode`\@=\active
\catcode`\@=11

\renewcommand{\@cite}[2]{[{{\bf #1}\if@tempswa , #2\fi}]}
\renewcommand{\@biblabel}[1]{[{\bf #1}]\hfill}
\catcode`\@=12

\newtheorem{thm}{Theorem}[section]
\newtheorem{lm}[thm]{Lemma}
\newtheorem{prop}[thm]{Proposition}

\theoremstyle{remark}
\newtheorem{rmk}[thm]{Remark}

\theoremstyle{definition}
\newtheorem{ex}[thm]{Example}
\newtheorem{df}{Definition}

\newcommand {\ah}{{\mathfrak a}}

\newcommand {\g}{{\mathfrak g}}
\newcommand {\h}{{\mathfrak h}}

\newcommand {\es}{{\mathfrak s}}
\newcommand {\te}{{\mathfrak t}}

\newcommand {\slv}{{\mathfrak{sl}(\BV)}}

\newcommand {\spv}{{\mathfrak{sp}(\BV)}}

\newcommand {\sov}{{\mathfrak {so}}(\BV)}

\newcommand {\BV}{{\mathbb V}}

\newcommand {\BQ}{{\mathbb Q}}

\newcommand {\sfr}{\eus R}

\newcommand{\lb}{\lambda}

\renewcommand{\le}{\leqslant}
\renewcommand{\ge}{\geqslant}

\newcommand{\eus}{\EuScript}

\newcommand {\ad}{{\mathrm{ad}}}

\newcommand {\hot}{{\mathsf{ht}}}
\newcommand {\ind}{{\mathsf{ind}}}

\newcommand {\rk}{{\mathsf{rk}}}

\newcommand {\tr}{{\mathrm{tr\,}}}

\newcommand {\tri}{\mathfrak{sl}_2}

\newcommand {\GR}[2]{{\textrm{{\bf #1}}}_{#2}}

\newcommand {\ov}{\overline}

\newcommand {\vlb}{\BV_\lb}
\newcommand {\beq}{\begin{equation}}
\newcommand {\eeq}{\end{equation}}

\newcommand {\bbk}{\Bbbk}

\begin{document}
\setlength{\parskip}{3pt plus 5pt minus 0pt}
\hfill { {\scriptsize November 12, 2013}}
\vskip1ex

\title%
{The Dynkin index and $\tri$-subalgebras of simple Lie algebras}
\author[D.\,Panyushev]{Dmitri I.~Panyushev}
\address[]{
Institute for Information Transmission Problems of the R.A.S., Bol'shoi Karetnyi per. 19, Moscow 127994
\ Russia 
}
\email{panyushev@iitp.ru}
\subjclass[2010]{17B20 (17B10, 17B22)}
\maketitle

\section*{Introduction}
\noindent
The ground field $\bbk$ is algebraically closed and of characteristic zero.
Let $G$ be a connected semisimple algebraic group with Lie algebra $\g$. 
In~1952, Dynkin classified all semisimple subalgebras of semisimple Lie algebras~\cite{dy}.
As a tool to distinguish different (non-conjugate) embeddings of the same algebra, Dynkin introduced 
the {\it index of a homomorphism of simple Lie algebras}. It will be convenient for us to split this
into the notions of (1) the index of a simple subalgebra of a simple Lie algebra and (2) the index of
a representation of a simple Lie algebra. 
After Mal'cev and Kostant, it is known that the conjugacy classes of the $\tri$-subalgebras of $\g$ are in 
a one-to-one correspondence with the nonzero nilpotent $G$-orbits in $\g$~\cite[3.4]{CM}. Therefore, 
one can define the index of a nilpotent element (orbit) as the Dynkin index of any associated 
$\tri$-subalgebra. As nilpotent orbits 
are related to the variety of intriguing problems in representation theory,
the indices of $\tri$-subalgebras of $\g$ are most interesting for us.
A simple Lie algebra has three distinguished nilpotent orbits: the principal (regular), subregular, and the 
minimal  ones. It was noticed by Dynkin that in the last case the corresponding $\tri$-index equals $1$ 
(cf.~\cite[Theorem\,2.4]{dy}). In~\cite{D-ind}, we gave a general formula for the index of a principal 
$\tri$-subalgebra of $\g$. 

This note can be regarded as a continuation of \cite{D-ind}.
Here we provide simple formulae for the index of all nilpotent orbits ($\tri$-subalgebras) in the classical 
Lie algebras~(Theorem~\ref{thm:all-classical}) and a new formula for the index of the principal 
$\tri$~(Theorem~\ref{thm:new-main}).
Then we compute the difference, $D$, of the indices of principal and subregular $\tri$-subalgebras. Our formula for $D$ involves some data related to the McKay correspondence for $\g$, see Theorem~\ref{thm:difference} and Eq.~\eqref{eq:another-D}. The index of a simple subalgebra $\es$ of $\g$,
$\ind(\es\hookrightarrow\g)$, can be computed via any non-trivial representation of $\g$, and taking different representations of $\g$, one gets different formal expression for $\ind(\es\hookrightarrow\g)$.
For $\es\simeq\tri$ and classical $\g$, we obtain essentially different formulae using the simplest and adjoint representations of $\g$, and the Jordan normal form of nonzero nilpotent elements of $\es$. This yields three series of interesting combinatorial identities parameterised by partitions, see 
Section~\ref{subs:classical}. We also prove that the index of a nilpotent orbit strictly decreases under the passage to the boundary of orbits (Proposition~\ref{cor:ind-ubyvanie}).

\section{The Dynkin indices of representations and subalgebras} 
\label{sect:di}

\noindent
Let $\g$ be a simple finite-dimensional Lie algebra of rank $n$. Let $\te$ be a Cartan
subalgebra, and $\Delta$ the set of roots of $\te$ in $\g$. Choose a set of positive roots
$\Delta^+$ in $\Delta$. Let $\Pi$ be the set of simple roots and $\theta$ the highest root in
$\Delta^+$. As usual, $\rho=\frac{1}{2}\sum_{\gamma>0}\gamma$.
The $\BQ$-span of all roots is a $\BQ$-subspace of $\te^*$, denoted $\eus E$.
Following Dynkin, we normalise a non-degenerate invariant symmetric bilinear form
$(\ ,\ )_\g$ on $\g$ 
as follows. The restriction of $(\ ,\ )_\g$ to $\te$ is 
non-degenerate, hence it induces the isomorphism of $\te$ and $\te^*$ and 
a non-degenerate bilinear form on $\eus E$. We then require that $(\theta,\theta)_\g=2$,
i.e.,  $(\beta,\beta)_\g=2$ for any {long} root $\beta$ in $\Delta$.

\begin{df}[Dynkin~{\cite[\S\,2]{dy}}]  \label{def:main} 
Let $\phi:\es\to\g$ be a homomorphism of simple Lie algebras. For $x,y\in\es$, the bilinear form $(x,y)\mapsto
(\phi(x),\phi(y))_\g$ is proportional to $(\ ,\ )_\es$ and 
the index of $\phi$ is defined by the equality 
$
   (\phi(x),\phi(y))_\g=\ind(\es\stackrel{\phi}{\to}\g){\cdot}(x,y)_\es, \ x,y\in\es$.
\end{df}
\noindent 
\textbullet \quad In particular, if $\es$ is a simple subalgebra of  $\g$, then
the {\it Dynkin index of $\es$ in $\g$\/} 
is   
\[
   \ind(\es\hookrightarrow\g):=\displaystyle \frac{(x,x)_\g}{(x,x)_\es}, \quad  x\in\es . 
\]
\noindent 
\textbullet \quad 
If $\nu: \g\to \mathfrak{sl}(M)$  is a representation of $\g$, then the {\it Dynkin index of
the representation\/} $\nu$, denoted $\ind_D(\g,M)$ or $\ind_D(\g,\nu)$, is defined by
\begin{equation}   \label{eq:ind_D}
  \ind_D(\g,M):=\ind(\g\stackrel{\nu}{\to}\mathfrak{sl}(M)) .
\end{equation}

\noindent
It is not hard to verify that, for the simple Lie algebra $\mathfrak{sl}(M)$, the normalised bilinear form
is given by  $(x,x)_{\slv}= \tr(x^2)$, $x\in\mathfrak{sl}(M)$.
Therefore, a more explicit expression for the Dynkin index of a representation 
$\nu: \g\to \mathfrak{sl}(M)$ is
\[ 
  \ind_D(\g,M)=\frac{\tr \bigl(\nu(x)^2\bigr)}{(x,x)_\g} .
\] 
The following properties easily follow from the definition:

{\sf Multiplicativity}: \ 
 If $\h\subset\es\subset\g$ are simple Lie algebras, then 
\[
\ind(\h\hookrightarrow\es){\cdot}\ind(\es\hookrightarrow\g)=\ind(\h\hookrightarrow\g) .
\]

{\sf Additivity}:  \ $\ind_D(\g, M_1\oplus M_2)=\ind_D(\g, M_1)+\ind_D(\g, M_2)$. \\
It is therefore sufficient to determine $\ind_D(\g, \cdot )$ for the irreducible representations.

\begin{thm}[Dynkin, \protect{\cite[Theorem\,2.5]{dy}}]   \label{thm:d-1952}
Let $\vlb$ be a  simple finite-dimensional $\g$-module with highest weight $\lb$. 
Then 
\[
   \ind_D(\g, \vlb)=\frac{\dim \vlb}{\dim\g} (\lb,\lb+2\rho)_\g .
\]
\end{thm}

\noindent
Although it is not obvious from the definition, the Dynkin index of a homomorphism is an 
integer~\cite[Theorem\,2.2]{dy}. Dynkin's original proof relied on classification results.
In 1954,  he gave a better proof that is based on  a topological interpretation of the index.
A short algebraic proof is given in \cite[Ch.\,I, \S 3.10]{on}.

Conversely,  the index of a simple subalgebra can be expressed via indices of representations. 
By the multiplicativity of index and Eq.~\eqref{eq:ind_D}, 
for a simple subalgebra $\es\subset\g$ and a non-trivial representation 
$\nu:\g\to \mathfrak{sl}(M)$, we have
\beq     \label{eq:various-V}
     \ind(\es\hookrightarrow\g)=
     \frac{\ind(\es\hookrightarrow \mathfrak{sl}(M))}{\ind(\g\hookrightarrow \mathfrak{sl}(M))}=
     \frac{\ind_D(\es,M)}{\ind_D(\g,M)} .
\eeq
A nice feature of this formula is that one can use various $M$ to compute the index of a given 
subalgebra.

\begin{ex}   \label{ex:1}
{\ }\phantom{\ }\\
(1) \ Let $\sfr_d$ be the simple $\tri$-module of dimension $d+1$. Then
$\ind_D(\tri,\sfr_d)= \genfrac{(}{)}{0pt}{}{d+2}{3}$.

\noindent
(2) \ Recall that $\theta$ is the highest root in $\Delta^+$.
By Theorem~\ref{thm:d-1952}, 
\[
  \ind_D(\g,\ad_\g)=(\theta,\theta+2\rho)_\g=
(\theta,\theta)_\g(1+(\rho,\theta^\vee)_\g)=2(1+(\rho,\theta^\vee)_\g) .
\]
\end{ex}
\noindent
Note that $(\rho,\theta^\vee)_\g$ does not depend on the normalisation of the bilinear form on 
$\eus E$. The integer $1+(\rho,\theta^\vee)_\g$ is customary called the {\it dual Coxeter number\/}
of $\g$, and we denote it by $h^*(\g)$. 
Thus, $\ind_D(\g,\ad_\g)=2h^*(\g)$.
In the simply-laced case, $h^*(\g)=h(\g)$---the usual
Coxeter number. For the other simple Lie algebras, we have
$h^*(\GR{B}{n})=2n{-}1$, $h^*(\GR{C}{n})=n{+}1$, $h^*(\GR{F}{4})=9$,
$h^*(\GR{G}{2})=4$.
Applying this to Eq.~\eqref{eq:various-V} with $M=\g$ and $\nu=\ad_\g$, we obtain
\beq    \label{eq:ind-subalg2}
    \ind(\es\hookrightarrow\g)=\frac{1}{2h^*(\g)}{\cdot}\ind_D(\es,\g) \ .
\eeq
More generally, we have
\begin{lm}   \label{lm:new}
If\/ $\es\subset\g$ are simple Lie algebras and $\nu:\g\to\mathfrak{sl}(M)$ is a representation, then
\[
  \ind_D(\es,M)=\frac{1}{2h^*(\g)}{\cdot}\ind_D(\es,\g){\cdot}\ind_D(\g,M).
\] 
\end{lm}
\begin{proof}
By the multiplicativity and  Eq.~\eqref{eq:ind-subalg2}, we have 
\[
  \ind_D(\es,M)=\ind(\es\hookrightarrow\g){\cdot}\ind(\g\hookrightarrow\mathfrak{sl}(M))=
  \frac{1}{2h^*(\g)}{\cdot}\ind_D(\es,\g){\cdot}\ind_D(\g,M) . \qedhere
\]
\end{proof}
\begin{rmk}   \label{rem:strange}
The ``strange formula'' of Freudenthal-de~Vries relates the scalar square of $\rho$ with $\dim\g$.
If $\langle\ ,\ \rangle$ is the {\it canonical\/} bilinear form on $\eus E$ with respect to $\Delta$, then
$\langle\rho,\rho\rangle=\dim\g/24$ \cite[47.11]{FdV}. The canonical bilinear form is characterised by the 
property that $\langle\gamma,\gamma\rangle=1/h^*(\g)$ for a {\it long\/} root $\gamma\in\Delta$. It follows 
that if $(\ ,\ )$ is any nonzero $W$-invariant bilinear form on $\eus E$ and $(\gamma,\gamma)=c$, then 
$(\rho, \rho)=\displaystyle\frac{\dim\g}{24}h^*(\g)c$.
\end{rmk}

\section{The index of  $\tri$-subalgebras and combinatorial identities} 
\label{sect:all-sl(2)}

\noindent
If $e\in\g$ is nonzero and nilpotent, then there exists a subalgebra $\ah\subset\g$ such that 
$\ah\simeq\tri$ and $e\in\ah$ (Morozov, Jacobson)\cite[3.3]{CM}. All $\tri$-subalgebras associated with 
a given $e$ are $G_e$-conjugate and we write $\GR{A}{1}(e)$ for such a subalgebra. In this section, 
we give explicit formulae for the indices $\ind(\GR{A}{1}(e)\hookrightarrow \g)$ and some applications.

Let $\g(\BV)$ be a classical simple Lie algebra (i.e., one of $\slv$, $\spv$, $\sov$). The nilpotent 
elements (orbits) in $\g(\BV)$ are parameterised by partitions of $\dim\BV$, and we give the formulae 
in terms of partitions.
For $e\in\g(\BV)$, let $\boldsymbol{\lb}(e)=(\lb_1,\lb_2,\dots)$ be the corresponding partition.
For $\spv$ or $\sov$, $\boldsymbol{\lb}(e)$ satisfies certain parity conditions~\cite{hess},\cite[5.1]{CM}, which are 
immaterial at the moment. And, of course, $\dim\BV$ is even in the symplectic case.

\begin{thm}   \label{thm:all-classical}   
For a nonzero nilpotent $e\in\g(\BV)$, with partition $\boldsymbol{\lb}(e)$, we have
\begin{itemize}
\item[\sf (i)] \ $\ind(\GR{A}{1}(e)\hookrightarrow \slv)=\ind(\GR{A}{1}(e)\hookrightarrow \spv)=
\sum_i \genfrac{(}{)}{0pt}{}{\lb_i+1}{3}$;
\item[\sf (ii)] \ $\ind(\GR{A}{1}(e)\hookrightarrow \sov)=
\frac{1}{2}\sum_i \genfrac{(}{)}{0pt}{}{\lb_i+1}{3}$.
\end{itemize}
\end{thm}
\begin{proof}
In all cases, we have $\BV\vert_{\GR{A}{1}(e)}=\oplus_i \sfr_{\lb_i-1}$.

(i) By formulae of Section~\ref{sect:di}, we have
\[
  \ind(\GR{A}{1}(e)\hookrightarrow \slv)=\ind_D(\GR{A}{1}(e),\BV)=
  \sum_i \ind_D(\GR{A}{1}(e),\sfr_{\lb_i-1})=\sum_i \genfrac{(}{)}{0pt}{}{\lb_i+1}{3}.
\] 
By the multiplicativity of the index, 
\[
   \ind(\GR{A}{1}(e)\hookrightarrow \slv)=\ind(\GR{A}{1}(e)\hookrightarrow \spv){\cdot}
   \ind(\spv\hookrightarrow \slv).
\]
Using 
Theorem~\ref{thm:d-1952}, one easily computes that 
$\ind(\spv\hookrightarrow \slv)=\ind_D(\spv,\BV)=1$.

(ii) Likewise, we use the fact that $\ind(\sov\hookrightarrow \slv)=\ind_D(\sov,\BV)=2$.
\end{proof}

For the exceptional Lie algebras, Dynkin already computed the index for all $\tri$-subalgebras 
\cite[Tables\,16--20]{dy}. His calculations can be verified as follows.
{\sl First}, for any nilpotent element $e\in\g$, the Jordan normal formal of $e$ in the 
simplest representation of $\g$ is determined in~\cite{law95}. {\sl Second}, using Theorem~\ref{thm:d-1952}, one obtains that the indices of the embeddings associated with the simplest representations of exceptional Lie algebras are:
\begin{center}
\begin{tabular}{lll}
$\ind(\GR{E}{6}\hookrightarrow \mathfrak{sl}_{27})=6$; &
$\ind(\GR{E}{7}\hookrightarrow \mathfrak{sp}_{56})=12$; &
$\ind(\GR{E}{8}\hookrightarrow \mathfrak{s0}_{248})=30$; \\
$\ind(\GR{F}{4}\hookrightarrow \mathfrak{so}_{26})=3$;  &
$\ind(\GR{G}{2}\hookrightarrow \mathfrak{so}_{7})=1$. & 
\end{tabular}
\end{center}
Combining these data with formulae of Theorem~\ref{thm:all-classical}, one readily computes the 
indices of all $\tri$-subalgebras.

\begin{prop}   \label{cor:ind-ubyvanie}
If $e,e'\in\g$ are nilpotent and $Ge'\subset\ov{Ge}\setminus Ge$, then 
\[
\ind(\GR{A}{1}(e')\hookrightarrow \g) <\ind(\GR{A}{1}(e)\hookrightarrow \g) .
\]
\end{prop}
\begin{proof} First, we prove this for $\g=\slv$, and then derive the general assertion.

1) $\g=\slv$. It suffices to consider the case in which $Ge'$ is dense in an irreducible component of
$\ov{Ge}\setminus Ge$. 

Here $\boldsymbol{\lb}(e')$ is obtained from $\boldsymbol{\lb}(e)$ via one
of the following procedures. If $\lb_i\ge \lb_{i+1}+2$, then 
$(\dots,\lb_i,\lb_{i+1},\dots)$ can be replaced with $(\dots,\lb_i-1,\lb_{i+1}+1,\dots)$. Or, a fragment
$(\dots, a+1,\underbrace{a,\dots,a}_k, a-1,\dots)$ in $\boldsymbol{\lb}(e)$ can be replaced with
$(\dots, \underbrace{a,\dots,a}_{k+2},\dots)$ \cite[Prop.\,3.9]{hess}. In both cases, one sees that the RHS in 
Theorem~\ref{thm:all-classical}(i) strictly decreases.

2) For an arbitrary simple $\g$, we consider a non-trivial representation $\nu:\g\to\slv$. If 
$Ge'\subset\ov{Ge}\setminus Ge$, then $SL(\BV)e'\subset \ov{SL(\BV)e}$. By a result of 
Richardson~\cite{ri67}, each irreducible component of $SL(\BV)e\cap\g$ is a (nilpotent) $G$-orbit. This 
also implies that $SL(\BV)e'\ne SL(\BV)e$. Hence
\[
    \ind(\GR{A}{1}(e')\hookrightarrow \g)=
    \frac{\ind(\GR{A}{1}(e')\hookrightarrow \slv)}{\ind(\g\hookrightarrow\slv)}<
    \frac{\ind(\GR{A}{1}(e)\hookrightarrow \slv)}{\ind(\g\hookrightarrow\slv)}=
    \ind(\GR{A}{1}(e)\hookrightarrow \g).  \qedhere
\]
\end{proof}

The  index of a subalgebra can be used for obtaining non-trivial combinatorial identities. 
Taking different $\g$-modules $M$  in Eq.~\eqref{eq:various-V} yields different 
expressions for $\ind(\es\hookrightarrow\g)$. If $\g=\g(\BV)$, then $\ind(\es\hookrightarrow\g)$
can be related to $\ind_D(\es,\BV)$ and 
there are two natural choices of test representations: 
the simplest representation, $M=\BV$, and the adjoint representation, $M=\g(\BV)$. 
Alternatively,  one can apply
Lemma~\ref{lm:new} to $\g=\g(\BV)$ and $M=\BV$. Anyway, the output is as follows:

\textbullet\quad If $\g=\slv$, then $\nu=\mathsf{id}$, $\ind_D(\slv,\BV)=1$, $h^*(\slv)=\dim\BV$, and
\beq    \label{eq:N1}
  \ind_D(\es,\BV)=\frac{\ind_D(\es,\slv)}{2\dim\BV} .
\eeq

\textbullet\quad If $\g=\spv$ and $\nu: \spv\to \slv$, then $\ind_D(\spv,\BV)=1$, 
$h^*(\spv)=\frac{1}{2}\dim\BV+1$, and
\beq    \label{eq:N2}
  \ind_D(\es,\BV)=\frac{\ind_D(\es,\spv)}{\dim\BV+2} .
\eeq

\textbullet\quad If $\g=\sov$ and $\nu: \sov\to \slv$, then $\ind_D(\sov,\BV)=2$, 
$h^*(\sov)=\dim\BV-2$, and
\beq    \label{eq:N3}
  \ind_D(\es,\BV)=\frac{\ind_D(\es,\sov)}{\dim\BV-2} .
\eeq

\subsection{Combinatorial identities related to $\g(\BV)$ and $\es\simeq\tri$}   \label{subs:classical} 
\leavevmode\par
If $\es \simeq\tri$ and a nonzero nilpotent element of $\es$ has the Jordan normal form with partition
$\boldsymbol{\lb}=(\lb_1,\lb_2,\dots)$, then $\sum_i\lb_i=\dim\BV$ and 
$\BV\vert_\es=\oplus_i \sfr_{\lb_i-1}$. In particular,
$\ind_D(\es,\BV)=\sum_i \genfrac{(}{)}{0pt}{}{\lb_i+1}{3}$, regardless of the type of $\g(\BV)$.
For each $\g(\BV)$, we use below the simple relation between the $\g(\BV)$-modules $\BV$ and $\g(\BV)$. 

1) $\g=\slv$. 
Using the Clebsch-Gordan formula, we obtain
\[
\mathfrak{gl}(\BV)\vert_\es=\BV\otimes\BV^*\vert_\es=\bigoplus_{i,j} \bigl(\sfr_{\lb_i-1}\otimes \sfr_{\lb_j-1}\bigr)=
\bigoplus_{i,j}\bigoplus_{k=0}^{\min\{\lb_i-1,\lb_j-1\}} \sfr_{\lb_i+\lb_j-2-2k} .
\]
Since $\mathfrak{gl}(\BV)$ and $\slv$ differ by a trivial $\g$-module, we have
$\ind_D(\es, \mathfrak{gl}(\BV))=\ind_D(\es,\slv)$. Then using Eq.~\eqref{eq:N1}, we obtain,  for an 
arbitrary partition $\boldsymbol{\lb}=(\lb_1,\lb_2,\dots)$, the identity
\[
   \sum_i \genfrac{(}{)}{0pt}{}{\lb_i+1}{3}=\frac{1}{2\sum_i \lb_i}
   \sum_{i,j}\sum_{k=0}^{\min\{\lb_i-1,\lb_j-1\}}\genfrac{(}{)}{0pt}{}{\lb_i+\lb_j-2k}{3} .
\]
In particular, for a principal nilpotent element $e\in\slv$, we have $\boldsymbol{\lb}(e)=(\dim\BV)=(N)$, 
and the identity reads
\[
   \genfrac{(}{)}{0pt}{}{N+1}{3}=\frac{1}{2N}\sum_{k=0}^{N-1} \genfrac{(}{)}{0pt}{}{2N-2k}{3} .
\]

2) $\g=\spv$. Here
\[
   \spv\vert_\es=\eus S^2(\BV\vert_\es)=\bigoplus_{i<j} \bigl(\sfr_{\lb_i-1}\otimes \sfr_{\lb_j-1}\bigr)
   \oplus \bigoplus_i \eus S^2(\sfr_{\lb_i-1}) \ 
\]
and $\eus S^2(\sfr_m)=\sfr_{2m}\oplus\sfr_{2m-4}\oplus\dots$ by a variation of the Clebsch-Gordan formula. Using Eq.~\eqref{eq:N2}, we then obtain
the ``symplectic identity'' 
\[
   \sum_i \genfrac{(}{)}{0pt}{}{\lb_i+1}{3}=\frac{1}{(\sum_i \lb_i)+2}\left(
   \sum_{i<j}\sum_{k=0}^{\lb_j-1}\genfrac{(}{)}{0pt}{}{\lb_i+\lb_j-2k}{3}+
    \sum_{i}\sum_{k=0}^{[\lb_i-1/2]}\genfrac{(}{)}{0pt}{}{2\lb_i-4k}{3}\right) ,
\]
where we use the fact that $\min\{\lb_i-1,\lb_j-1\}=\lb_j-1$ if $i<j$.
For instance, $\boldsymbol{\lb}(e)=(\dim\BV)=(2n)$ for a principal nilpotent element $e\in\spv$, and
the identity reads
\[
   \genfrac{(}{)}{0pt}{}{2n+1}{3}=\frac{1}{2n+2}\sum_{k=0}^{n-1} \genfrac{(}{)}{0pt}{}{4n-4k}{3} .
\]

3) $\g=\sov$. Here $\sov\simeq \wedge^2(\BV)$
and $\wedge^2(\sfr_m)=\sfr_{2m-2}\oplus\sfr_{2m-6}\oplus\dots$. Then using 
Eq.~\eqref{eq:N3} we obtain 
the ``orthogonal identity''
\[
   \sum_i \genfrac{(}{)}{0pt}{}{\lb_i+1}{3}=
   \frac{1}{(\sum_i \lb_i)-2}\left(
   \sum_{i<j}\sum_{k=0}^{\lb_j-1}\genfrac{(}{)}{0pt}{}{\lb_i+\lb_j-2k}{3}+
    \sum_{i}\sum_{k=1}^{[\lb_i/2]}\genfrac{(}{)}{0pt}{}{2\lb_i+2-4k}{3}\right) .
\]
In particular, if $\dim\BV=2n$, then 
$\boldsymbol{\lb}(e)=(2n-1,1)$ for a principal nilpotent element $e\in\sov$, and the identity is
\[
   \genfrac{(}{)}{0pt}{}{2n}{3}=\frac{1}{2n-2}\left( \genfrac{(}{)}{0pt}{}{2n}{3}+
   \sum_{k=1}^{n-1} \genfrac{(}{)}{0pt}{}{4n-4k}{3}\right) .
\]

\section{On the index of principal and subregular $\tri$-subalgebras} 
\label{sect:algebra}

\noindent 
If $e\in\g$ is a  {\it principal\/} (= {\it regular\/}) nilpotent element, then the corresponding 
$\tri$-subalgebras are also called {\it principal}. We refer to \cite[n.\,29]{dy} and 
\cite[Sect.\,5]{ko59} for properties of principal $\tri$-subalgebras.
The set of non-regular nilpotent elements contains a dense $G$-orbit~\cite[4.2]{CM}. The elements of this orbit and
corresponding $\tri$-subalgebras are said to be {\it subregular}. 
Write $(\tri)^{pr}$ (resp. $(\tri)^{sub}$) for a principal (resp. subregular) 
$\tri$-subalgebra of $\g$.
In \cite{D-ind}, we obtained a uniform expression for 
$\ind((\tri)^{pr}\hookrightarrow \g)$.  To recall it, we need some notation.

Let $\theta_s$ denote the short dominant
root in $\Delta^+$. (In the simply-laced case, we assume that  $\theta=\theta_s$.)
Set $r=\|\theta\|^2/\|\theta_s\|^2 \in\{1,2,3\}$.
Along with $\g$, we also consider the Langlands
dual  algebra $\g^\vee$, which is determined by the dual root system $\Delta^\vee$.
Since the Weyl groups of $\g$ and $\g^\vee$ are isomorphic, we have $h(\g)=h(\g^\vee)$.
However, the dual Coxeter numbers can be different (cf. $\GR{B}{n}$ and $\GR{C}{n}$). 
The half-sum of the positive roots for $\g^\vee$ is
\[
   \rho^\vee:=\frac{1}{2}\sum_{\gamma>0} \gamma^\vee=\sum_{\gamma>0}\frac{\gamma}{
(\gamma,\gamma)_\g} .
\] 
It is well-known (and easily verified) that $(\rho^\vee,\gamma)_\g=\hot(\gamma)$ for any 
$\gamma\in\Delta^+$. (This equality does not depend on the normalisation of a bilinear form on $\eus E$.)
It follows that $h^*(\g^\vee)=1+(\rho^\vee, \theta_s)=1+\hot(\theta_s)$. Our first uniform expression is

\begin{thm}[{\cite[Theorem\,3.2]{D-ind}}]   \label{thm:old-main}
$\ind((\tri)^{pr}\hookrightarrow \g)=\displaystyle\frac{\dim\g}{6}h^*(\g^\vee) r$.
\end{thm}

Below, we give yet another expression for this index. Let $\Delta^+_l$ (resp. $\Delta^+_s$) be the set
of long (resp. short) positive roots. In the simply-laced case, all roots are assumed to be short and 
$r=1$.

\begin{thm}   \label{thm:new-main}
 $\ind((\tri)^{pr}\hookrightarrow \g)=2(\rho^\vee,\rho^\vee)_\g=
 \displaystyle\sum_{\gamma\in\Delta^+_l}\hot(\gamma)+r\sum_{\gamma\in\Delta^+_s}\hot(\gamma)$.
\end{thm}
\begin{proof} In view of our choice of the form $(\ ,\ )_\g$, we have
\[
   2\rho^\vee=\sum_{\gamma\in\Delta^+}\frac{2\gamma}{(\gamma,\gamma)_\g}=
   \sum_{\gamma\in\Delta^+_l}\gamma+r \sum_{\mu\in\Delta^+_s}\mu .
\]
Consequently,  
\[
   2(\rho^\vee,\rho^\vee)_\g=
   (\rho^\vee, \sum_{\gamma\in\Delta^+_l}\gamma+r \sum_{\mu\in\Delta^+_s}\mu)_\g=
   \sum_{\gamma\in\Delta^+_l}\hot(\gamma)+r\sum_{\gamma\in\Delta^+_s}\hot(\mu) ,
\]
which yields the second equality. 

Now, we obtain another expression for $(\rho^\vee,\rho^\vee)_\g$ applying the ``strange formula'' of
Freudenthal-de Vries to $\Delta^\vee$ and $\g^\vee$, cf. Remark~\ref{rem:strange}. If $\mu\in\Delta_s$, then $\mu^\vee$ is a long
root in $\Delta^\vee$ and $(\mu^\vee,\mu^\vee)_\g=2r$. Therefore,
$2(\rho^\vee,\rho^\vee)_\g=\displaystyle2\frac{\dim(\g^\vee)}{24}2r h^*(\g^\vee)=
\frac{\dim\g}{6}r h^*(\g^\vee)$, which is exactly the index of $(\tri)^{pr}$.
\end{proof}

\begin{rmk}
It was noticed in \cite{D-ind} that the index of $(\tri)^{pr}$ is preserved under the unfolding procedure
$\g\leadsto \tilde\g$ applied to the multiply laced Dynkin diagram,
the four pairs $(\g,\tilde\g)$ being $(\GR{C}{n},\GR{A}{2n-1})$,  $(\GR{B}{n},\GR{D}{n+1})$, 
$(\GR{F}{4},\GR{E}{6})$, $(\GR{G}{2},\GR{D}{4})$. Using Theorem~\ref{thm:new-main}, we may look at
this coincidence from another angle.
Let $\tilde\Delta$ be the root system of $\tilde\g$ with respect to a Cartan subalgebra $\tilde\te$.
The embedding $\te\to\tilde\te$ induces a surjective map $\pi:\tilde\Delta^+\to \Delta^+$ such that
$\pi^{-1}(\Delta^+_l)\to \Delta^+_l$ is one-to-one and $\#\pi^{-1}(\gamma)=r$ for $\gamma\in\Delta^+_s$.
Furthermore, $\pi$ is height-preserving. Thus, we get
the natural equality
$
      \sum_{\gamma\in\Delta^+_l}\hot(\gamma)+r\sum_{\gamma\in\Delta^+_s}\hot(\mu)=
      \sum_{\tilde\gamma\in\tilde\Delta^+}\hot(\tilde\gamma)$, which again "explains" the coincidence of two indices.
\end{rmk}

Our next goal is to provide a simple uniform expression for the difference of the indices of subalgebras
$(\tri)^{pr}$ and $(\tri)^{sub}$. To this end, we need the relationship between the structure of 
$\g$ as the module over $(\tri)^{pr}$ or $(\tri)^{sub}$, see e.g.~\cite[Ch.\,7]{slodowy}.
Let $m_1,\dots,m_n$ be the exponents of $\g$. As was shown by Kostant~\cite{ko59}, 
\beq             \label{eq:restr-pr}
    \g\vert_{(\tri)^{pr}}=\bigoplus_{i=1}^n \sfr_{2m_i} .
\eeq
To deal with the subregular $\tri$-subalgebras, 
we may assume that $n=\rk(\g)\ge 2$ and also $1=m_1<m_2\le \dots \le m_{n-1}<m_n=h(\g)-1$.
Then 
\beq             \label{eq:restr-sub}
    \g\vert_{(\tri)^{sub}}=\left(\bigoplus_{i=1}^{n-1} \sfr_{2m_i}\right)\oplus 
    \sfr_{a-2}\oplus\sfr_{b-2}\oplus\sfr_{h(\g)-2} ,
\eeq
where $a+b=h(\g)+2$. Assume that $a\le b$ and note that $(a,b,h(\g))$ are just $(w_r,w_{r+1},w_{r+2})$ in \cite[p.\,112]{slodowy}.
Below, we write $h$ and $h^*$ for $h(\g)$ and $h^*(\g)$, respectively.
\begin{thm}            \label{thm:difference}
$D:=\ind((\tri)^{pr}\hookrightarrow \g)-\ind((\tri)^{sub}\hookrightarrow \g)=
\displaystyle\frac{h}{h^*}\bigl(\genfrac{(}{)}{0pt}{}{h}{2}+\frac{(a-2)(b-2)}{4}\bigr)$.
\end{thm}
\begin{proof}  If $\g\vert_{\tri}=\oplus_j \sfr_{n_j}$, then Eq.~\eqref{eq:ind-subalg2} shows that
$\ind(\tri\hookrightarrow\g)=\frac{1}{2h^*} \sum_j \genfrac{(}{)}{0pt}{}{n_j+2}{3}$. Therefore,
by Eq.~\eqref{eq:restr-pr} and \eqref{eq:restr-sub}, the difference $D$ equals
\[
   \frac{1}{2h^*}(\genfrac{(}{)}{0pt}{}{2h}{3}-\genfrac{(}{)}{0pt}{}{h}{3}-\genfrac{(}{)}{0pt}{}{a}{3}-
   \genfrac{(}{)}{0pt}{}{b}{3}) .
\]
Then routine transformations, where we repeatedly use the relation $(a-1)+(b-1)=h$, simplify this expression to the desired form. For instance,  we first transform
$\genfrac{(}{)}{0pt}{}{a}{3}+\genfrac{(}{)}{0pt}{}{b}{3}$ into $\frac{h}{6}(h^2-3(a-1)(b-1)-1)$, etc.
\end{proof}

\noindent
In the following table, we gather the relevant data for all simple Lie algebras.

\begin{center}
\begin{tabular}{|c|ccccccccc|}
\hline
$\g$ & $\GR{A}{n}$ & $\GR{B}{n}$ & $\GR{C}{n}, n{\ge}3$ & $\GR{D}{n}, n{\ge}4$ & $\GR{E}{6}$ & $\GR{E}{7}$ &  $\GR{E}{8}$ & $\GR{F}{4}$ & $\GR{G}{2}$ \\ \hline 
$\ind((\tri)^{pr}{\hookrightarrow}\g)$ & \rule{0pt}{2.7ex}$\genfrac{(}{)}{0pt}{}{n+2}{3}$ &
$\frac{1}{2}\genfrac{(}{)}{0pt}{}{2n+2}{3}$ 
& $\genfrac{(}{)}{0pt}{}{2n+1}{3}$ & $\frac{1}{2}\genfrac{(}{)}{0pt}{}{2n}{3}$ 
&
$156$ & $399$ & $1240$ & $156$ & $28$ \\
$D$ & $\rule{0pt}{2.7ex}\genfrac{(}{)}{0pt}{}{n+1}{2}$ & $2n^2$ & $4n(n{-}1)$ & $2n(n{-}2)$ & 
$72$ & $168$ & $480$ & $96$ & $24$ \\  
$a$ & $2$ & $2$ & $4$ & $4$ & $6$ & $8$ & $12$ & $6$ & $4$ \\
$b$ & $n{+}1$ & $2n$ & $2n{-}2$ & $2n{-}4$ & $8$ & $12$ & $20$ & $8$ & $4$ \\ 
$D/b{\cdot}\rk(\g)$ & $1/2$ & $1$ & $2$ & $1$ & $3/2$ & $2$ & $3$ & $3$ & $3$ \\
\hline
 \end{tabular}
\end{center}
\begin{rmk}
The numbers $(a,b)$ frequently occur in the study of the McKay correspondence and finite 
subgroups of $SL_2$, see e.g.~\cite{ko06}. Recall that Slodowy associates a finite subgroup of $SL_2$ to any $\g$ (not only of type {\sf A-D-E}) \cite[6.2]{slodowy}. Let $\tilde\Gamma\subset SL_2$ be the finite subgroup corresponding to $\g$.
Then (i) $ab/2=\#\tilde\Gamma$, (ii)  $\{a,b,h\}$ are the degrees of basic invariants for the associated
2-dimensional representation of $\tilde\Gamma$, and (iii) the Poincar\'e series of this ring of invariants is
$\displaystyle \frac{1+T^h}{(1-T^a)(1-T^b)}$. Using the first relation, one can also write
\beq   \label{eq:another-D}
D=\frac{h}{h^*}{\cdot}\frac{h(h-2)+\#\tilde\Gamma}{2} .
\eeq
\end{rmk}

\begin{rmk}    \label{rem:empirical}
Let us point out some curious observations related to $D$.

\textbullet \quad It is always true that $D \le 2h{\cdot}\rk(\g)$, and the equality holds if and only if
$\g$ is of type $\GR{G}{2},\GR{F}{4},\GR{E}{8}$. Furthermore, if $h$ is even (which only excludes the case of $\GR{A}{2n}$), then $D/\rk(\g)$ is an integer.

\textbullet \quad It is always true that $D \le 3b{\cdot}\rk(\g)$, and the equality holds if and only if
$\g$ is of type $\GR{G}{2},\GR{F}{4},\GR{E}{8}$. Moreover, for each classical series, the ratio
$D/b{\cdot}\rk(\g)$ is constant. 

It might be interesting to find an explanation for these properties and understand the meaning of the constant $D/b{\cdot}\rk(\g)$.
\end{rmk}
\noindent 
{\small {\bf Acknowledgements.}
I would like to thank the Mathematisches Institut der Friedrich-Schiller-Universit\"at (Jena) for the warm hospitality during the preparation of this article.}

\end{document}